\documentclass[12pt, a4paper]{article}
\usepackage{amsmath}
\usepackage{comment}
\usepackage{amsfonts}
\usepackage{amssymb}
\usepackage{epsfig}
\usepackage{graphicx}
\usepackage{color}
\usepackage{amsthm}
\usepackage{enumerate}
\usepackage [latin1]{inputenc}
\usepackage[numbers, sort&compress]{natbib}
\usepackage{url}

\newtheorem{thm}{Theorem}[section]
\newtheorem{lem}[thm]{Lemma}
\newtheorem{cor}[thm]{Corollary}
\newtheorem{pro}[thm]{Proposition}

\newtheorem{con}[thm]{Conjecture}
\newtheorem{prob}[thm]{Problem}

\theoremstyle{definition}

\newtheorem{remark}[thm]{Remark}
\graphicspath{{figures/}}

\usepackage{url}

\usepackage{authblk}

\long\def\delete#1{}

\usepackage{xcolor}
\usepackage[normalem]{ulem}

\begin{document}
\openup 0.5\jot

\title{Graphs with large maximum forcing number}
\author[1,3]{Qianqian Liu\thanks{E-mail: \texttt{liuqq2023@imut.edu.cn.}}}
\author[2]{Ajit A. Diwan}
\author[1]{Heping Zhang\footnote{The corresponding author. E-mail: \texttt{zhanghp@lzu.edu.cn.}}}

\affil[1]{\small School of Mathematics and Statistics, Lanzhou University, Lanzhou, Gansu 730000, China}
\affil[2]{\small Department of Computer Science and Engineering, Indian Institute of Technology Bombay, Powai, Mumbai 400076, India}
\affil[3]{\small College of Science, Inner Mongolia University of Technology, Hohhot, Inner Mongolia 010010, China}

\date{}

\maketitle

\setlength{\baselineskip}{20pt}
\noindent {\bf Abstract}:
For  a graph $G$ with  order $2n$ and  a perfect matching, let $f(G)$ and $F(G)$ denote the minimum and maximum forcing number of $G$ respectively. 
Then $0\leq f(G)\leq F(G)\leq n-1$.  Liu and Zhang \cite{liu} ever proposed a conjecture: $e(G)\geq \frac{n^2}{n-F(G)}$, where $e(G)$ denotes the number of edges of $G$. In this paper we confirm this conjecture and obtain $F(G)\leq n-\frac{n^2}{e(G)}$.
If $F(G)=n-1$, Liu and Zhang \cite{47} proved that  any two perfect matchings of $G$ can be obtained from each other by a series of matching switches along 4-cycles. 
If $G$ is bipartite and $F(G)\geq n-k$, $1\leq k\leq n-1$, we show that  any two perfect matchings of $G$ can be obtained from each other by a series of matching switches along even cycles of length at most $2(k+1)$. Finally, we ask whether  $f(G)\geq \lceil\frac{n}{k}\rceil-1$ holds for such bipartite graphs $G$, and give positive answers  for the  cases  $k=1,2$. Further we show all minimum forcing numbers of the bipartite graphs $G$ of order $2n$ and with
$F(G)=n-2$ form an integer interval $[\lfloor\frac{n}{2}\rfloor, n-2]$.

\vspace{2mm} \noindent{\it Keywords}: Perfect matching; Maximum forcing number; Bipartite graph; Matching switch
\vspace{2mm}


\section{\normalsize Introduction}

We consider only finite and simple graphs. Let $G$ be a graph with vertex set $V(G)$ and edge set $E(G)$. The \emph{order} of $G$ is the number of vertices in $G$.
A \emph{matching} $M$ in $G$ is an edge subset  such that no two edges of $M$  have an endvertex in common. The vertices incident to the edges of $M$ are called \emph{$M$}-\emph{saturated}, and the others are called \emph{$M$}-\emph{unsaturated}.
A \emph{perfect matching} of $G$ is a matching that saturates every vertex.

Let $G$ be a graph with a perfect matching $M$. A subset of $M$ is called a \emph{forcing set} of $M$ in $G$ if it is not contained in any other perfect matching of $G$. The smallest cardinality of a forcing set of $M$ is called the \emph{forcing number} of $M$, denoted by $f(G,M)$. The concept of forcing number was originally introduced by Klein and Randi\'{c} \cite{3} under the name of the innate degree of freedom, which plays an important role in resonance theory.
The \emph{minimum} and \emph{maximum} \emph{forcing numbers} of $G$ is the minimum and maximum values of forcing numbers over all perfect matchings of $G$, denoted by $f(G)$ and $F(G)$, respectively. Afshani et al. \cite{5} showed that computing the minimum forcing number is NP-complete for bipartite graphs with the maximum degree 4. However the complexity of computing the maximum forcing number of a graph is unknown. For  forcing numbers of perfect matchings of graphs and some related topics, readers may  refer to two surveys \cite{1,zhanghe}.

Let $G$ be a graph of order $2n$ and with a perfect matching. Then $0\leq f(G)\leq F(G)\leq n-1$.
Liu and Zhang \cite{liu} obtained a lower bound of  $e(G)$ by using maximum forcing number $F(G)$, where $e(G)$ denotes the number of edges of $G$. To improve the bound they proposed the following conjecture.


\begin{con} \cite{liu} \label{con}Let $G$ be a graph with order $2n$ and a perfect matching. Then
$$e(G)\geq \frac{n^2}{n-F(G)},  \text{ equivalently, }  F(G)\leq n-\frac{n^2}{e(G)}.$$
\end{con}

For some special cases $F(G)\leq \frac{n}{2}$ and $F(G)=1,2$, the conjecture was already confirmed. In this paper we completely confirm this conjecture.  Applying this bound on the $d$-dimensional hypercube $Q_d$, we obtain that $F(Q_d)\leq (1-\frac{1}{d})2^{d-1}$; see Section 3.

 Che and Chen \cite{1} proposed a problem: How to characterize the graphs $G$ with $f(G)=n-1$. For bipartite graphs, they \cite{10} obtained that $f(G)=n-1$ if and only if $G$ is complete bipartite graph $K_{n,n}$. Moreover,
Liu and Zhang \cite{47} solved the problem by obtaining the following result.

\begin{thm}\cite{47} Let $G$ be a graph of order $2n$. Then $f(G)=n-1$ if and only if $G$ is a complete multipartite graph with each partite set having size no more than $n$ or a graph obtained by adding arbitrary additional edges in one partite set to complete bipartite graph $K_{n,n}$.
\end{thm}

However, it is difficult to characterize the graphs $G$ with order $2n$ and $F(G)=n-1$, a larger class of graphs. Some interesting properties have been  already revealed.

\begin{thm}\cite{47}\label{cont} If $G$ is a graph of order $2n$ and with $F(G)=n-1$, then the forcing spectrum (the set of forcing numbers of all perfect matchings) is continuous.
\end{thm}

\begin{thm}\cite{47} All minimum forcing numbers of the  graphs $G$ with order $2n$ and  $F(G)=n-1$ form an
integer interval $[\lfloor\frac{n}{2}\rfloor, n-1]$.
\end{thm}

To prove Theorem \ref{cont} they showed that any two perfect matchings of $G$ can be obtained from each other by applying a series of matching switches along 4-cycles.
In Sections 4 and 5 of this paper, we consider all bipartite graphs $G$ of order $2n$ and with $F(G)\geq n-k$ for $1\leq k\leq n$.
In section 4 we prove that any two perfect matchings of  $G$ can be obtained from each other by applying a series  of matching switches along even cycles of length at most $2(k+1)$. 
In Section 5, we propose a problem: does   $f(G)\geq \lceil\frac{n}{k}\rceil-1$ always hold?  We give positive answers  for the  cases  $k=1,2$. Further we show all minimum forcing numbers of the bipartite graphs $G$ of order $2n$ and with
$F(G)=n-2$ form an integer interval $[\lfloor\frac{n}{2}\rfloor, n-2]$.

\section{\normalsize Preliminary}
In this section, we give some notations and prove an important lemma.

For sets $A$ and $B$, the symmetric difference of $A$ and $B$ is defined as $A\oplus B:=(A\setminus B)\cup (B\setminus A)$.  For a perfect matching $M$ of a graph $G$, a path or cycle is \emph{$M$-alternating} if its edges appear alternately in $M$ and off $M$.
If $C$ is an $M$-alternating cycle of length $2l$ for $l\geq 2$, then the symmetric difference   $M\oplus E(C) $, is called a \emph{matching switch} on $M$ along $C$ of length $2l$ (simply {\it matching $l$-switch}), to get another perfect matching of $G$.

The \emph{degree} of a vertex $v$ in $G$, denoted by $d_G(v)$, is the number of edges incident to $v$.
The minimum and maximum degree of $G$ are denoted by $\delta(G)$ and $\Delta(G)$, respectively. The \emph{neighborhood} of $v$, denoted by $N_G(v)$, is the set of vertices adjacent to $v$. We denote a cycle with $l(\geq 3)$ vertices by $C_l$. A graph $G$ is called \emph{$C_l$-free} if $G$ contains no induced subgraphs isomorphic to $C_l$. A \emph{chord} of a cycle $C$ is an edge not in $C$ whose endvertices lie in $C$.

For a  vertex subset $S$ of $G$, $G[S]$ denotes the subgraph of $G$ \emph{induced by} $S$ whose vertex set is $S$ and edge set consisting of all edges of $G$ with both endvertices in $S$.
In a word, $G[S]$ is the subgraph $G-V(G)\setminus S$, obtained from $G$ by deleting all vertices in $V(G)\setminus S$ and their incident edges.
For an edge subset $T$ of $G$, we use $V(T)$ to denote the set of all endvertices of edges in $T$.
For two vertex subsets $V_1$ and $V_2$ of $G$, we denote by $E(V_1,V_2)$ the set of edges of $G$ with one endvertex in $V_1$ and the other in $V_2$.

By the definition of a forcing set, we obtain the following lemma.
\begin{lem}\label{2.1}Let $G$ be a graph of order $2n$ for $1\leq k\leq n$ and with a perfect matching $M$. Then $f(G,M)\geq n-k$ if and only if $G[V(T)]$ contains an
$M$-alternating cycle for any subset $T$ of $M$ with size at least $k+1$.
\end{lem}
\begin{proof}The lemma is trivial for $n=k$. Let $n\geq k+1$. If $f(G,M)\geq n-k$, then $M\setminus T$ is not a forcing set of $M$ for any subset $T$ of $M$ with $|T|\geq k+1$. By the definition, $M\setminus T$ is contained in at least two perfect matchings of $G$, and $G[V(T)]$ has at least two perfect matchings. Thus $G[V(T)]$ contains an $M$-alternating cycle. Conversely, for any subset $T$ of $M$ with size less than $n-k$, $|M\setminus T|\geq k+1$. By the assumption, $G[V(M\setminus T)]$ contains an $M$-alternating cycle and has at least two perfect matchings. Thus $T$ is not a forcing set of $M$, and $f(G,M)\geq n-k$.
\end{proof}

Let $G$ be a bipartite graph of order $2n$ and with $F(G)=n-1$. Then there is a perfect matching $M=\{u_iv_i|1\leq i\leq n\}$ of $G$ with $f(G,M)=n-1$. By Lemma \ref{2.1}, $G[\{u_i,v_i,u_j,v_j\}]$ contains an $M$-alternating cycle for $1\leq i< j\leq n$. That is to say, $\{u_iv_j,v_iu_j\}\subseteq E(G)$. Thus, $G$ is complete bipartite graph $K_{n,n}$. Hence, we obtain the following result.
\begin{cor}\label{cor2.2} Let $G$ be a bipartite graph of order $2n$. Then $F(G)=n-1$ if and only if $G$ is  complete bipartite graph $K_{n,n}$.
\end{cor}

\section{\normalsize A solution to Conjecture \ref{con}}

To confirm Conjecture \ref{con} we obtain  the following result.
\begin{thm} \label{edgesthm}
Let $G$ be a graph of order $2n$ and with a perfect matching. If $F(G)\geq n-k$ where $1\leq k\leq n$, then $e(G)\geq \frac{n^2}{k}$.
\end{thm}


\begin{proof}Let $M$ be a perfect matching of $G$ with $f(G,M)=F(G)\geq n-k$.
We will proceed by induction on $k$. For $k=1$, by Lemma \ref{2.1}, the induced subgraph by all endvertices of any two edges of $M$ contains an $M$-alternating cycle. That is, there are at least two edges not in $M$ between the endvertices of the two edges of $M$. Thus, $\delta(G)\geq n$ and $e(G)\geq n^2$.

For $k\geq 2$, let $w$ be a vertex of $G$ with the minimum degree $d$. Let $$G':= G- N_G(w)\cup\{s|st\in M, t\in N_G(w)\}, M'=M\cap E(G'),\text{ and } t_d=n-|M'|.$$
Then $M'$ is a perfect matching of graph $G'$ that has $2(n-t_d)$ vertices, and $t_d$ equals the number of edges in $M$ incident with a vertex in $N(w)$. So $t_d\leq d$, and the equality holds if $G$ is bipartite.

Next we consider the following two cases.
If $n-t_d\leq k-1$, then $t_d\geq n-k+1$. By Degree-Sum formula, we have $$e(G)\geq nd\geq nt_d\geq n(n-k+1).$$
Since $n\geq k$ and $k\geq 1$, we have $n-k\geq 0$ and $1-\frac{1}{k}\geq 0$. Thus
\begin{eqnarray*}
n(n-k+1)-\frac{n^2}{k}&\geq&n(n-k+1-\frac{n}{k})\\
            &\geq&n[n(1-\frac{1}{k})-k(1-\frac{1}{k})]\\
            &\geq&n(n-k)(1-\frac{1}{k})\geq 0.
\end{eqnarray*}

If $n-t_d\geq k$,  we will prove that $f(G',M')\geq (n-t_d)-(k-1)$. Let $S'$ be a subset of $M'$ with size $k$ and $e$ be an edge of $M$ incident to $w$. By Lemma \ref{2.1}, $G[V(S'\cup \{e\})]$ contains an $M$-alternating cycle. Since $w$ is of degree 1 in $G[V(S'\cup \{e\})]$, $e$ is not contained in this $M$-alternating cycle. Thus this $M$-alternating cycle is entirely contained in $G'[V(S')]$. By Lemma \ref{2.1}, $F(G')\geq f(G',M')\geq (n-t_d)-(k-1)$. By the induction hypothesis, $e(G')\geq \frac{(n-t_d)^2}{k-1}$. Since there are at least $dt_d$ edges in $G$ but not in $G'$, $$e(G)\geq dt_d+\frac{(n-t_d)^2}{k-1}\geq t_d^2+\frac{(n-t_d)^2}{k-1}=\frac{k}{k-1}t_d^2-\frac{2n}{k-1}t_d+\frac{n^2}{k-1},$$ which is a quadratic function of $t_d$ that reaches the minimum value $\frac{n^2}{k}$ at $t_d=\frac{n}{k}$.
\end{proof}

The following corollary shows that  Conjecture \ref{con} is true, which gives an upper bound of the maximum forcing number.
\begin{cor} \label{corhy}
Let $G$ be a graph of order $2n$ and with a perfect matching. Then $$F(G)\leq n-\frac{n^2}{e(G)}.$$
\end{cor}
\begin{proof}Since $G$ has a perfect matching, we have $0\leq F(G)\leq n-1$. That is, $1\leq n-F(G)\leq n$. Letting $k:=n-F(G)$, we have $F(G)=n-k$. By Theorem \ref{edgesthm} we have $e(G)\geq \frac{n^2}{k}=\frac{n^2}{n-F(G)}$. Equivalently, we can have that $F(G)\leq n-\frac{n^2}{e(G)}.$
\end{proof}

\begin{remark}\label{rmk3} The   bounds in Theorem  \ref{edgesthm} and Corollary \ref{corhy} are sharp. Let $n,k$ be two positive integers such that $n\equiv 0$ (mod $k$). Let $G'$ be a minimal graph of order $\frac{2n}{k}$ and with $F(G')=\frac{n}{k}-1$. Then $G'$ is $\frac{n}{k}$-regular by Lemma 2.3 (ii) in \cite{47}.  Let $G$ be the disjoint union of $k$ copies of $G'$. Then $F(G)=\sum_{i=1}^kF(G')=n-k$ and $G$ has exactly $\frac{n^2}{k}$ edges. Especially, we can choose $G'$ as complete bipartite graph $K_{\frac{n}{k},\frac{n}{k}}.$
\end{remark}

\begin{cor}\label{corhyy}
Let $G$ be a $r$-regular graph  with a perfect matching. Then $$F(G)\leq \frac{r-1}{2r}|V(G)|.$$
\end{cor}

\begin{proof}
Since $G$ is $r$-regular,  $e(G)=\frac{r|V(G)|}{2}$. By Corollary \ref{corhy}, we have
 $$F(G)\leq \frac{|V(G)|}{2}(1-\frac{1}{r})=\frac{r-1}{2r}|V(G)|.$$
\end{proof}
Let $Q_d$ be a $d$-dimensional hypercube where $d\geq 1$.  We can immediately obtain an upper bound of $F(Q_d)$ by Corollary \ref{corhy} or \ref{corhyy}.
\begin{cor}
For $d\geq 1$, $$F(Q_d)\leq (1-\frac{1}{d})2^{d-1}.$$
\end{cor}
Alon ever showed that $F(Q_d)>c2^{d-1}$ for any constant $0<c<1$ and sufficiently large $d$ (see \cite{7}). Using Alon's method,
Adams et al. \cite{4}  obtained a lower bound in a more general sense as follows.
$$F(Q_d)\geq (1-\frac{\log(2e)}{\log d})2^{d-1}, d\geq 2.$$

Combining  the above lower and upper bounds on $F(Q_d)$ we have
 $$(1-\frac{\log(2e)}{\log d})2^{d-1}\leq F(Q_d)\leq (1-\frac{1}{d})2^{d-1}, d\geq 2,$$

\noindent which implies  an asymptotic behavior for the maximum forcing number $F(Q_d)$ as follows. \begin{cor} $F(Q_d)\sim 2^{d-1}$ as $d$ approaches to $\infty$.
\end{cor}
\section{\normalsize Switches of perfect matchings}
From this section we consider the bipartite graphs $G$ of order $2n$ and with $F(G)= n-k$ where $1\leq k\leq n$. If $k=n$, then $F(G)=0$ and $G$ has a unique perfect matching. The following theorem is our main result of this section.

\begin{thm}\label{qpp} Let $G$ be a bipartite graph of order $2n$ and with $F(G)\geq n-k$ for $1\leq k\leq n$. Then any two perfect matchings of $G$ can be obtained from each other by a finite sequence of matching 2-, 3-, \dots, or $(k+1)$-switches.
\end{thm}
\begin{proof}We will prove this theorem by induction on $n$.
For $n=k$, we have $F(G)\geq 0$. If $F(G)=0$, then $G$ has a unique perfect matching, and the result is trivial. Otherwise, $G$ has at least two perfect matchings, say $M$ and $M'$. Then $M\oplus M'$ forms some disjoint $M$- ($M'$-) alternating cycles of length at most $2n$. Switching along such $M'$-alternating cycles on $M'$, we obtain perfect matching $M$.

Next   suppose   $n\geq k+1$. Let $M=\{u_iv_i|1 \leq i \leq n\}$ be a perfect matching of $G$ with $f(G,M)\geq n-k$ and let $M'$ be any other perfect matching of $G$. It suffices to prove that $M$ can be obtained from $M'$ by a finite sequence of matching 2-, 3-, $\dots$, or $(k+1)$-switches.
Then we have the following claims.

\smallskip
{\textbf{Claim 1.} If $M\cap M'\neq\emptyset$, then $M$ can be obtained from $M'$ by a finite sequence of matching 2-, 3-, \dots, or $(k+1)$-switches.}
\smallskip

Suppose that $M\cap M'$ contains an edge $u_iv_i$ for some $1\leq i\leq n$.
Let $H=G-\{u_i,v_i\}$ and $M_H=M\setminus\{u_iv_i\}$, $M_H'=M'\setminus\{u_iv_i\}$. Then $H$ has $2(n-1)$ vertices, $M_H$ and $M_H'$ are two distinct perfect matchings of $H$. We will prove that $f(H,M_H)\geq (n-1)-k$. Let $T_H$ be a forcing set of $M_H$. By the definition, $T_H$ is contained in only perfect matching $M_H$ of $H$. Hence, $T_H\cup\{u_iv_i\}$ is contained in only perfect matching $M$ of $G$. So $T_H\cup\{u_iv_i\}$ is a forcing set of $M$. Since $|T_H\cup\{u_iv_i\}|\geq f(G,M)\geq n-k$, we have $|T_H|\geq n-k-1$. Thus, $F(H)\geq f(H,M_H)\geq (n-1)-k$.
By the induction hypothesis, $M_H$ can be obtained from $M_H'$ by repeatedly applying finite times of matching 2-, 3-, \dots, $(k+1)$-switches. Thus, $M$ is also obtained from $M'$ by the same series of matching 2-, 3-, \dots, $(k+1)$-switches.

 \smallskip
\smallskip
{\textbf{Claim 2.} If $M\oplus M'$ contains a cycle of length $2l$ for some $2\leq l\leq k+1$, then $M$ can be obtained from $M'$ by a finite sequence of matching 2-, 3-, \dots, or $(k+1)$-switches.}
\smallskip
\smallskip

Let $C$ be a cycle of length $2l$ in $M\oplus M'$. Switching along the $M'$-alternating cycle $C$, we can obtain a perfect matching $M''$, which has at least $l$ edges intersected with $M$. By
Claim  1, $M$ can be obtained from $M''$ by a finite sequence of matching 2-, 3-, \dots, $(k+1)$-switches. Since $M''$ is obtained from $M'$ by a matching $l$-switch, $M$ can be obtained from $M'$ by repeatedly applying finite times of matching 2-, 3-, \dots, $(k+1)$-switches, and the claim holds.

By Claims 1 and 2, next we only consider the case when $M$ and $M'$ are disjoint and every cycle in $M\oplus M'$ is of length at least $2(k+2)$.
Take a cycle $C=u_0v_0u_1v_1\dots u_{l-1}v_{l-1}u_0$ in $M\oplus M'$ for some $l\geq k+2$ such that $u_iv_i$ is an edge in $M$ and $v_iu_{i+1}$ is
an edge in $M'$ for $0\leq i\leq l-1$ where $i+1$ is considered modulo $l$. We will prove that $C$ has a chord. Taking a matching of $k+1$ edges as $\{v_1u_1,v_0u_0,v_{l-1}u_{l-1},\dots,v_{l-(k-1)}u_{l-(k-1)}\}\subseteq M$, by Lemma \ref{2.1}, the induced subgraph $G[\{v_1,u_1,v_0,u_0,v_{l-1},u_{l-1},\dots,v_{l-(k-1)},u_{l-(k-1)}\}]$ contains an $M$-alternating cycle. Since the intersection of this induced subgraph and $C$ forms an $M$-alternating path, we obtain that $C$ has a chord, say $v_iu_{i+j}$, with the subscripts modulo $l$.
\begin{figure}[h]
\centering
\includegraphics[height=4.5cm,width=6.5cm]{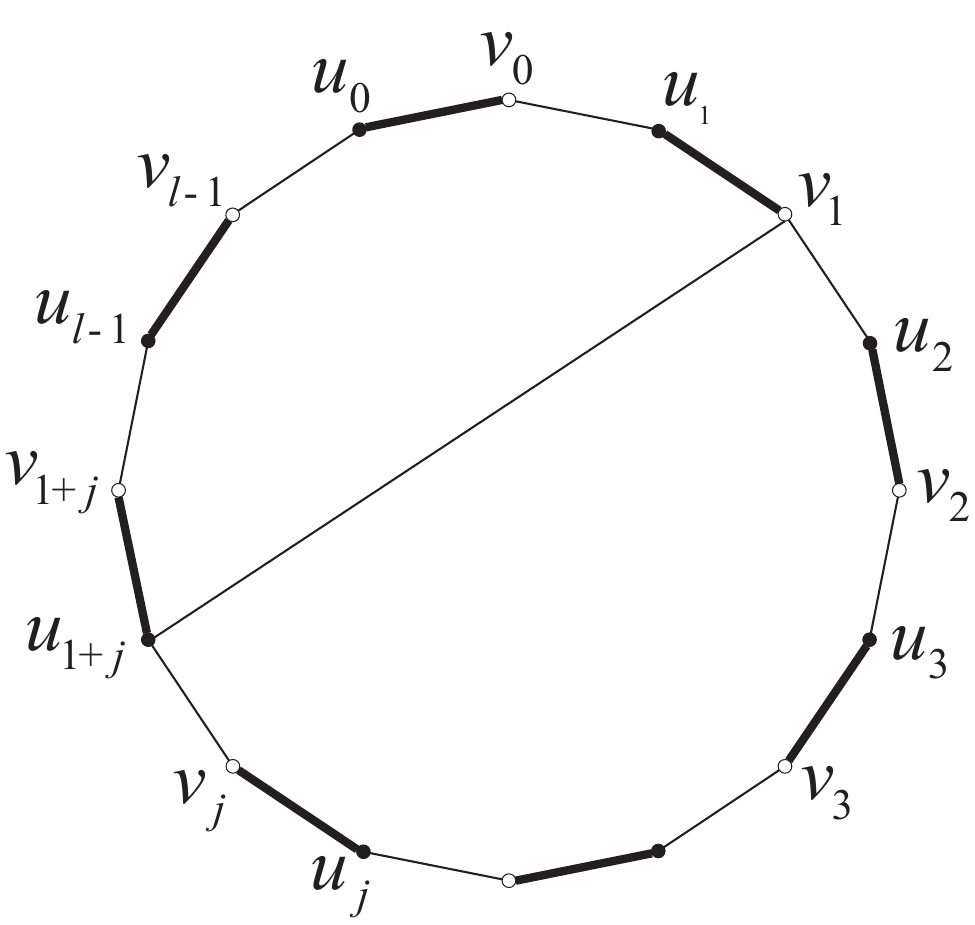}
\caption{\label{cyclec}$M$- and $M'$-alternating cycle $C$ with a chord $v_1u_{1+j}$.}
\end{figure}

Choose $i$ and $j$ so that $j$ is minimum possible. Without loss of generality, we assume that $i=1$ (see Fig. \ref{cyclec}). Since $v_1u_{1+j}$ is a chord, we have $2\leq j\leq l-1$.
If $j\leq k+1$, then $v_1u_2v_2\dots u_jv_ju_{j+1}v_1$ is an $M'$-alternating cycle containing $j$ edges of $M'$. Switching along this cycle gives a perfect matching $M''$ that has $j-1$ edges in common with $M$. By Claim 1,
$M$ can be obtained from $M''$ by a finite sequence of matching 2-, 3-, \dots, or $(k+1)$-switches. Since $M''$ is obtained from $M'$ by a matching $j$-switch, we obtain that $M$ is obtained from $M'$ by a finite sequence of matching 2-, 3-, \dots, or $(k+1)$-switches.
Therefore we assume that $j\geq k+2\geq 3$.

Next by finding at least $j-k-1$ $M$-alternating cycles, we will show that $C$ has at least $j-k-1$ extra chords.
Consider the induced subgraph $G':=G[\{u_3,v_3,u_4,v_4,\dots,u_j,v_j\}]$. In $G'$, $v_q$ ($3\leq q\leq j$) can be incident only to the chords of the form $v_qu_m$ where $3\leq m < q$. Otherwise, there exists a chord $v_qu_s$ where $q+2\leq s\leq j$. Note that $s-q\leq j-q\leq j-3< j$, which contradicts the choice of $i$ and $j$. Taking a subset $\{u_1v_1,u_3v_3,\dots, u_jv_j\}$ of $M$ with size at least $k+1$, by Lemma \ref{2.1}, $G[\{u_1,v_1,u_3,v_3,\dots, u_j,v_j\}]$ contains an $M$-alternating cycle. Since $j$ is minimum, $v_1$ is of degree 1 in $G[\{u_1,v_1,u_3,v_3,\dots, u_j,v_j\}]$, and $u_1v_1$ is not contained in this $M$-alternating cycle. Thus, this $M$-alternating cycle is entirely contained in $G'$.

 \smallskip
\smallskip
{\textbf{Claim 3.} Any $M$-alternating cycle in $G'$ is of the form $u_mv_mu_{m+1}v_{m+1}\dots u_qv_qu_m$, where $3\leq m< q\leq j$.}
\smallskip
\smallskip

Let $C'$ be an $M$-alternating cycle in $G'$ and $m$ be the smallest subscript such that $u_mv_m\in E(C')$. Then $v_mu_y\notin E(C')$ for $3\leq y< m$. Otherwise, both $v_mu_y$ and $u_yv_y$ are contained in $E(C')$, which contradicts the minimality of $m$.
Now $\{v_mu_{m+1},u_{m+1}v_{m+1}\}\subseteq E(C')$. Take the longest $M$-alternating path $u_mv_mu_{m+1}v_{m+1}\cdots u_qv_q$ on $C'$. Then $v_qu_{q+1}\notin E(C')$. Since $v_q$ ($3\leq q\leq j$) can be incident only to the chords of $v_qu_z$ where $3\leq z < q$ in $G'$, there exists $3\leq z<q$ such that $v_qu_z\in E(C')$. Combining that $C'$ is an $M$-alternating cycle, we obtain that $u_z$ is distinct from the vertices in $\{u_{m+1},u_{m+2},\dots,u_q\}$. Thus, we have $z\leq m$.
By the minimality of $m$, we have $z=m$. Hence $C'=u_mv_mu_{m+1}v_{m+1}\dots u_qv_qu_m$ (see Fig. \ref{cycleform}).
Since the intersection of $C'$ and the path $u_3v_3\cdots u_jv_j$ is just the $M$-alternating path $u_mv_m\cdots u_qv_q$, we say that $C'$ starts at $u_m$ and ends at $v_q$.
\begin{figure}[h]
\centering
\includegraphics[height=5cm,width=6.5cm]{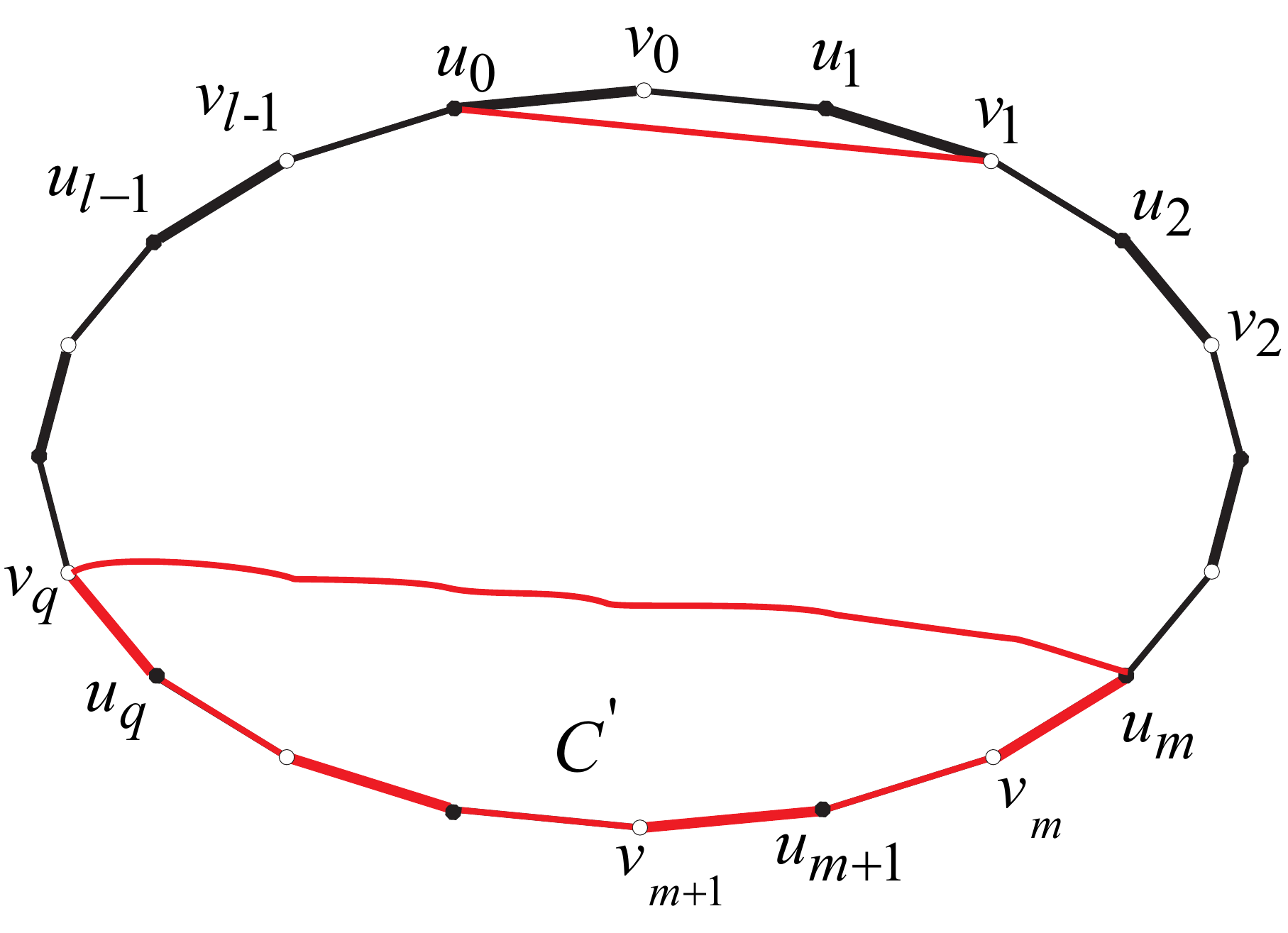}
\caption{\label{cycleform}The form of an $M$-alternating cycle $C'$ in $G'$.}
\end{figure}

Next we will prove that there exist at least $j-k-1\geq 1$ disjoint $M$-alternating cycles in $G'$. We choose such cycles greedily. Choose the first $M$-alternating cycle ends at $v_{l_1}$ such that $l_1$ is as small as possible. Having chosen such cycles ending at $v_{l_1}, v_{l_2},\dots, v_{l_r}$, choose $l_{r+1}$ to be the smallest index such that some cycle starts at $u_{m_{r+1}}$ for $m_{r+1} > l_r$ and ends at $v_{l_{r+1}}$. This process is continued until no more such cycles can be chosen. See Fig. \ref{cycleccc} for an example where the $M$-alternating cycles chosen are in red.
\begin{figure}[h]
\centering
\includegraphics[height=5cm,width=7cm]{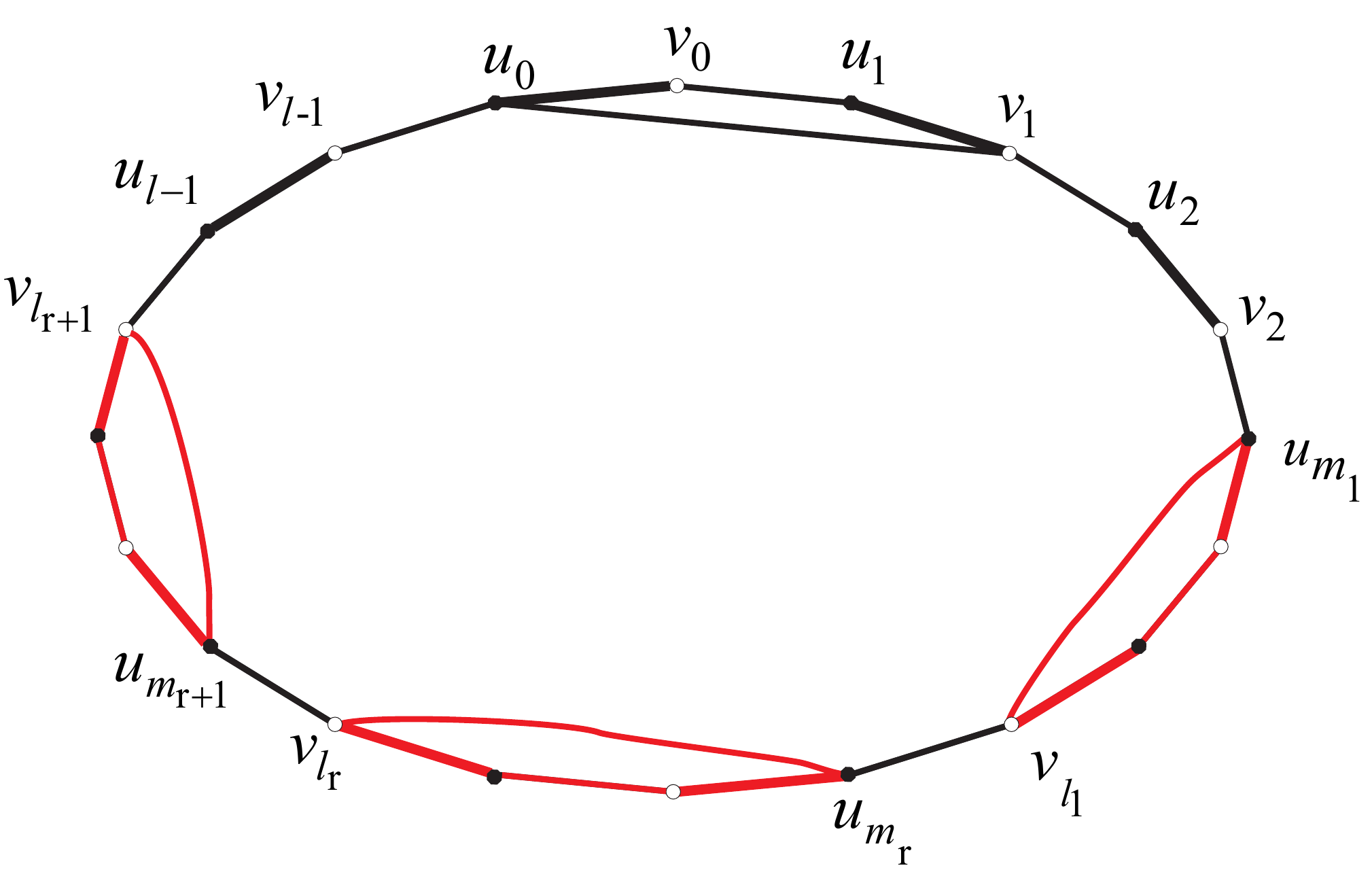}
\caption{\label{cycleccc}Constructed process of $M$-alternating cycles in $G'$.}
\end{figure}
Let $t$ denote the number of cycles selected and $m_1 < l_1 < m_2< l_2 <\cdots < m_t < l_t$ where the indices such that the $p$'th cycle starts at $u_{m_p}$ and ends at $v_{l_p}$.
We claim that $t\geq j-k-1$. Suppose $t\leq j-k-2$. Consider the induced subgraph $G'':=G[(V(G')\setminus\{u_{l_p},v_{l_p}|1\leq p \leq t\})\cup \{u_1,v_1\}]$.
Then $G''$ contains at least $(j-2)-(j-k-2)+1=k+1$ edges of $M$.  By Lemma \ref{2.1}, $G''$ contains an $M$-alternating cycle. Since $j$ is minimum such that $v_1u_{1+j}\in E(G)$, the vertex $v_1$ has degree 1 in $G''$. Thus any $M$-alternating cycle in $G''$ is contained in $G''-\{u_1,v_1\}$, which is an induced subgraph of $G'$. By Claim 3, any $M$-alternating cycle in $G'$ consists of a chord and an $M$-alternating path of $C$ between two endvertices of the chord. Hence, this $M$-alternating cycle should be selected, which is a contradiction.

Now we are to construct an $M'$-alternating cycle of length no more than $2(k+1)$ using these $t+1$ chords we have found and $t+1$ $M'$-alternating paths of $C$.
Let $v_{l_0}=v_1$ and $u_{m_{t+1}}=u_{j+1}$. Then these $t+1$ chords $\{v_1u_{j+1}\}\cup \{u_{m_p}v_{l_p}|1\leq p\leq t\}$ and $t+1$ $M'$-alternating paths $v_{l_p}u_{l_{p+1}}v_{l_{p+1}}u_{l_{p+2}}\cdots u_{m_{p+1}}$ for $0\leq p\leq t$ form an $ M'$-alternating cycle $C''$, which contains an edge $u_2v_2$ of $M$. See $C''$ in Fig. \ref{cyclet} where $t+1$ chords shown in red and $t+1$ $M'$-alternating paths shown in green.
\begin{figure}[h]
\centering
\includegraphics[height=5cm,width=7cm]{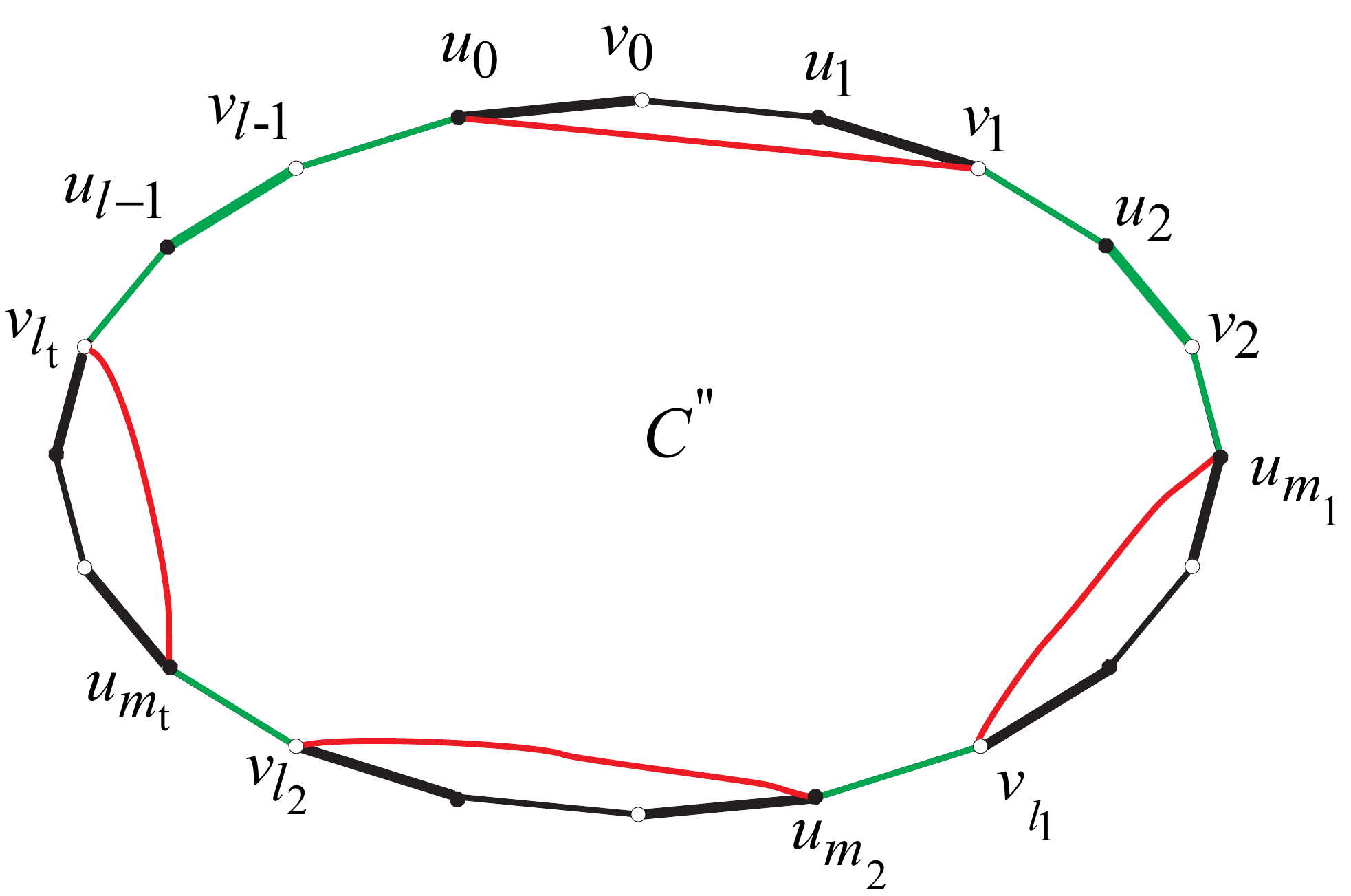}
\caption{\label{cyclet}An $M'$-alternating cycle consisting of $t+1$ chords and $t+1$ $M'$-alternating paths.}
\end{figure}

Since the $M'$-alternating path $v_1u_2v_2\cdots u_jv_ju_{j+1}$ contains exactly $j$ edges of $M'$, and each chord $u_{m_p}v_{l_p}$ decreases at least one edge in $M'$ where $1\leq p\leq t$, $C''$ contains at most $j-t\leq k+1$ edges of $M'$. Switching along $C''$ gives a perfect matching $M''$ so that $u_2v_2\in M''\cap M$. By Claim 1, $M$ can be obtained from $M''$ by a finite sequence of matching 2-, 3-, \dots, or $(k+1)$-switches. Noting that $M''$ is obtained from a matching switch along an $M'$-alternating cycle of length no more than $2(k+1)$, we obtain that $M$ can be obtained from $M'$ by a finite sequence of matching 2-, 3-, \dots, or $(k+1)$-switches.
\end{proof}

In the remianing of this section we make some remarks and applications of Theorem \ref{qpp}.
\begin{remark}\label{rmk1}
If we only apply matching 2-, 3-,\dots,$k$-switches, then the result in Theorem \ref{qpp} is not necessarily true. To show this point we construct a bipartite graph $G_{n,k}$ of order $2n$   as follows (see Fig. \ref{ep3}): $V(G_{n,k})=\{u_{i},v_i|1\leq i\leq n\},$ and  $$E(G_{n,k})=\{v_iu_{i},v_iu_{i+1}|1\leq i\leq k-1\}\cup \{u_1v_{k+1}\}\cup \{u_tv_s|k\leq t,s\leq n\}-\{u_kv_{k+1}\}.$$
\end{remark}
\begin{figure}[h]
\centering
\includegraphics[height=3.5cm,width=10cm]{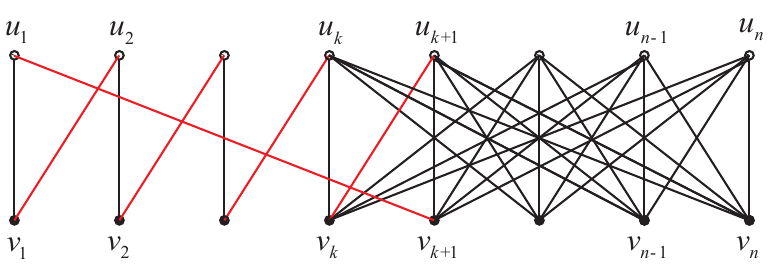}
\caption{\label{ep3}A bipartite graph $G_{n,k}$ for Remark \ref{rmk1}.} 
\end{figure}

\begin{pro}\label{rmk} (\expandafter{\romannumeral1}) For $2\leq k\leq n-1$, we have $F(G_{n,k})=n-k$.\\
(\expandafter{\romannumeral2}) Let $M=\{u_iv_i|1\leq i\leq n\}$ and $M'=\{v_iu_{i+1}|1\leq i\leq k\}\cup\{u_1v_{k+1}\}\cup\{u_{i}v_{i}|k+2\leq i\leq n\}$ be two perfect matchings of $G_{n,k}$. Then $M$ cannot be obtained from $M'$ by only applying matching 2-, 3-, \dots, or $k$-switches.
\end{pro}
\begin{proof}(\expandafter{\romannumeral1})
First we prove that $f(G_{n,k},M)\geq n-k$. Let $T$ be any subset of $M$ with size at least $k+1$. If $T$ contains at least two edges of $\{u_{k+1}v_{k+1},u_{k+2}v_{k+2},\dots,u_nv_n\}$, then $G_{n,k}[V(T)]$ contains an $M$-alternating cycle for $G_{n,k}[\{u_{k+1},v_{k+1},\dots,u_n,v_n\}]$ is a complete bipartite graph. Otherwise, $T$ contains exactly one edge of $\{u_{k+1}v_{k+1},u_{k+2}v_{k+2},\dots,u_nv_n\}$ and all edges of $\{u_1v_1,u_2v_2,\dots, u_kv_k\}$. If $u_{k+1}v_{k+1}\in T$, then $G_{n,k}[V(T)]$ itself forms an $M$-alternating cycle $v_{k+1}u_1v_1u_2v_2\cdots u_kv_ku_{k+1}v_{k+1}$. If $u_{k+1}v_{k+1}\notin T$ and $u_iv_i\in T$ for some $k+2\leq i\leq n$, then $G_{n,k}[V(T)]$ contains an $M$-alternating cycle $u_kv_ku_iv_iu_k$. By Lemma \ref{2.1}, we obtain that $F(G_{n,k})\geq f(G_{n,k},M)\geq n-k$.

Next we prove that $f(G_{n,k},F)\leq n-k$ for any perfect matching $F$ of $G_{n,k}$. Since $u_1$ is of degree 2, $F$ contains $u_1v_1$ or $u_1v_{k+1}$. If $u_1v_{k+1}\in F$, then $\{v_1u_2,v_2u_3,\dots,v_{k-1}u_{k}\}\subseteq F$. Since $G_{n,k}[\{u_1,v_{k+1},v_1,u_2,\dots,v_{k-1},u_{k}\}]$ contains no $F$-alternating cycle, $F\setminus \{u_1v_{k+1},v_1u_2,\dots,v_{k-1}u_{k}\}$ is a forcing set of $F$ and $f(G_{n,k},F)\leq n-k$. If $u_1v_1\in F$, then $\{u_2v_2,\dots,u_{k-1}v_{k-1}\}\subseteq F$. We assume that $u_kv_i\in F$ for some $i\in \{k,k+2, k+3,\dots,n\}$. Since $G_{n,k}[\{u_1,v_1,\dots,u_{k-1},v_{k-1},u_k,v_i\}]$ contains no $F$-alternating cycle, $F\setminus\{u_1v_1,u_2v_2,\dots,u_{k-1}v_{k-1},u_kv_i\}$ is a forcing set of $F$ and $f(G_{n,k},F)\leq n-k$. Therefore, we have $f(G_{n,k},F)\leq n-k$. By the arbitrariness of $F$, we have $F(G_{n,k})\leq n-k.$

(\expandafter{\romannumeral2})
If we want to switch $M'$ to $M$ by matching 2-, 3-, \dots, or $k$-switches, then we must switch $u_1v_{k+1}$ to another edge by some $M'$-alternating cycle of length at most $2k$. Since $u_i$ and $v_i$ are of degree two for $1\leq i\leq k-1$, any $M'$-alternating cycle must contain the $M'$-alternating path with $2k$ vertices $v_{k+1}u_1v_1u_2\cdots v_{k-1}u_k$. Since $u_kv_{k+1}\notin E(G_{n,k})$, the $M'$-alternating path cannot be contained in any $M'$-alternating cycle of length at most $2k$. Thus, there is no such matching 2-, 3-, \dots, or $k$-switches to change $M'$ to $M$.
\end{proof}

Obviously, $G_{n,k}$ is not $C_{2(k+1)}$-free. For $C_6$-, $C_8$-, \dots, and $C_{2(k+1)}$-free graphs, we can obtain a stronger result than Theorem \ref{qpp}.
\begin{cor}\label{qpp2} Let $G$ be a $C_6$-, $C_8$-, \dots, and $C_{2(k+1)}$-free bipartite graph of order $2n$, where $1\leq k\leq n$. If $F(G)=n-k$, then any two perfect matchings of $G$ can be obtained from each other by a finite sequence of matching 2-switches.
\end{cor}
\begin{proof}By Theorem \ref{qpp}, any two perfect matchings of $G$ can be obtained from each other by a finite sequence of matching 2-, 3-, \dots, or $(k+1)$-switches.

Take a perfect matching $M$ of $G$ and  an $M$-alternating cycle $C=u_1v_1u_2v_2\cdots u_lv_lu_1$, where $3\leq l\leq k+1$ and $u_rv_r\in M$ for $1\leq r\leq l$. First we prove that a matching $l$-switch on $M$ along $C$ can be obtained by matching switches along two cycles of length no more than  $2(l-1)$.
Since $G$ is $C_{2l}$-free, $C$ has a chord, say $u_iv_j$, where $1\leq i< j\leq l$. Let $C^1:=u_iv_i\cdots u_jv_ju_i$ be an $M$-alternating cycle and $C^2:=v_ju_{j+1}v_{j+1}\cdots u_lv_lu_1v_1\cdots u_iv_j$ be an $M\oplus E(C^1)$-alternating cycle, whose lengths are at most $2(l-1)$. Then $M\oplus E(C)=M\oplus E(C^1)\oplus E(C^2)$. That is, a matching $l$-switch  on $M$ along $C$ can be obtained from two consecutive matching switches on $M$ and $M\oplus E(C^1)$ along $C^1$ and $C^2$, respectively.

If $C^i$ has length at least 6 for some $i\in \{1,2\}$, then we can continue the process until all cycles obtained are of length 4. Therefore, we can replace all matching $l$-switches for $3\leq l\leq k+1$ by finite times of matching 2-switches.
\end{proof}

\begin{remark}\label{rmk2} Theorem \ref{qpp} does not hold for non-bipartite graphs where $k\geq  2$. 
To show this point, we define the following graphs.

  Given  integers $n\geq 4$ and $k\geq 2$. Assume that $\sum^{k}_{j=1} n_j=n$ is a partition of $n$ so that each part $n_i$ is at least two. For $1\leq i\leq k$, construct  graph $G^i$: the vertex set is $U_i\cup V_i$ where $U_i=\{u_{\sum^{i-1}_{j=1} n_j+1},u_{\sum^{i-1}_{j=1} n_j+2},\dots,u_{\sum^{i}_{j=1} n_j}\}$ and $V_i=\{v_{\sum^{i-1}_{j=1} n_j+1},v_{\sum^{i-1}_{j=1} n_j+2},\dots,v_{\sum^{i}_{j=1} n_j}\}$; Let $M^i=\{u_lv_l|\sum^{i-1}_{j=1} n_j+1\leq l\leq \sum^{i}_{j=1} n_j\}$ and $T_i=\{v_lu_{l+1}|\sum^{i-1}_{j=1} n_j+1\leq l\leq \sum^{i}_{j=1} n_j-1\}.$ Then the $G^i$ is obtained from the union of  complete graphs  $G^i[V_i]$ and $G^i[U_i]$ by adding $M^i\cup T_i$.

  Obviously $M^i$ is a perfect matching of $G^i$.
Since $G^i[\{v_j,u_j,v_s,u_s\}]$ contains an $M^i$-alternating cycle for any two edges $\{u_jv_j,u_sv_s\}\subseteq M^i$, by Lemma \ref{2.1}, $F(G^i,M^i)=n_i-1$ for $1\leq i\leq k$.

Further let $G$ be a larger graph obtained from the disjoint union of $G^1$, $G^2$, \dots, $G^k$ by adding the extra edges $S=\{v_{n}u_1\}\cup \{v_{\sum^{i-1}_{j=1} n_j}u_{\sum^{i-1}_{j=1} n_j+1}| 2\leq i\leq k\}$ (see Fig. \ref{example}). Then $M=M^1\cup M^2\cup\cdots \cup M^{k}$ and $M'=\bigcup_{i=1}^kT_i\cup S$ are two perfect matchings of $G$.
\end{remark}

\begin{figure}[h]
\centering
\includegraphics[height=3cm,width=14cm]{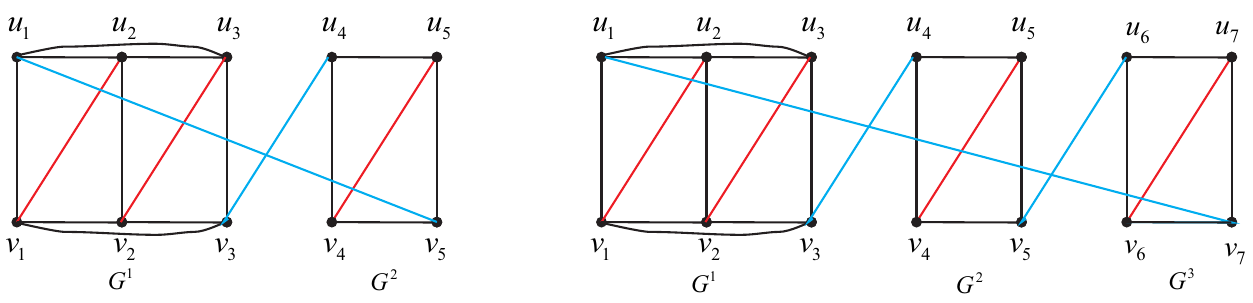}
\caption{\label{example}Graphs $G$ with perfect matchings $M$   and $M'$ for $k=2$ (left) and $k=3$ (right).}
\end{figure}

\begin{pro}\label{rmk2}  $F(G)= n-k$ and  $M$ cannot be obtained from $M'$ by applying only matching 2-, 3-, \dots, or $(k+1)$-switches.
\end{pro}

\begin{proof} 

(\expandafter{\romannumeral1}) First we have $$F(G)\geq f(G,M)\geq \sum_{i=1}^{k}f(G^i,M^i)=\sum_{i=1}^{k}(n_i-1)=n-k.$$ Next we are to prove that $f(G,F)\leq n-k$ for any perfect matching $F$ of $G$.  If $S\cap F\neq\emptyset$, then $S\subseteq F$ since each $G^i$ is of even order. If $S\subseteq F$, then $F\setminus S$ is a forcing set of $F$ and $f(G,F)\leq |F\setminus S|=n-k$. Otherwise, we have $S\cap F=\emptyset$. Then $F\cap E(G^i)$ is a perfect matching of $G^i$ for $1\leq i\leq k$. Let $s_i$ be one edge of $F\cap E(G^i)$. Since $u_{\sum^{i-1}_{j=1} n_j+1}$ and $v_{\sum^{i}_{j=1} n_j}$ is not adjacent for $1\leq i\leq k$, $G[\cup_{i=1}^k V(\{s_i\})]$ contains no $F$-alternating cycles. Thus, $F\setminus (\cup_{i=1}^ks_i)$ is a forcing set of $F$, and $f(G,F)\leq |F\setminus (\cup_{i=1}^ks_i)|=n-k$. By the arbitrariness of $F$, we have $F(G)\leq n-k$. Therefore, $F(G)=n-k$.

(\expandafter{\romannumeral2})
If we want to transform  $M'$ into $M$ by matching 2-, 3-, \dots, or $(k+1)$-switches, then we must switch $v_{n_1}u_{n_1+1}$ to another edge by some $M'$-alternating cycle of length at most $2(k+1)$. Since $G^i$ is of even order for $1\leq i\leq k$, all edges in $S$ are contained in this $M'$-alternating cycle. Since any two edges of $S$ are not adjacent, such $M'$-alternating cycle contains at least one edges of $T^i$ and two edges of $E(G^i)\setminus T^i$. Thus, such $M'$-alternating cycles are of length at least $4k\geq 2(k+2)$. Therefore, there is no such matching 2-, 3-, \dots, or $(k+1)$-switches to change $M'$ to $M$.
\end{proof}

Matching switch is an important tool to study the forcing spectrum of a graph, which collects the forcing numbers of all perfect matchings of $G$. 
Afshani et al. \cite{5} showed that a matching 2-switch on a perfect matching does not change the forcing number by more than 1. Theorem \ref{cont} was proved in this way.

For a plane elementary bipartite graph, Deng \cite{52} obtained that a matching $l$-switch on a perfect matching does not change the forcing number by more than $l-1$. In fact, this result can be extended to general graphs.
\begin{lem}\label{mt}Let $G$ be a graph with a perfect matching $M$, and $C_{2l}$ be an $M$-alternating cycle. Then $|f(G,M\oplus E(C_{2l}))-f(G,M)|\leq l-1$.
\end{lem}
\begin{proof}Suppose that $E(C_{2l})\cap M=\{e_1,e_3,\dots,e_{2l-1}\}$ and $E(C_{2l})\setminus M=\{e_2,e_4,\dots,e_{2l}\}$.
Let $S$ be a minimum forcing set of $M$. Then $S\cap E(C_{2l})\neq \emptyset$ and $G-V(S)$ has a unique perfect matching. Let $S'=S\cup \{e_2,e_4,\dots,e_{2l}\}\setminus \{e_1,e_3,\dots,e_{2l-1}\}$. Then $G-V(S')$ is an induced subgraph of $G-V(S)$, and thus $G-V(S')$ has a unique perfect matching. So $S'$ is a forcing set of the perfect matching $M\oplus E(C_{2l})$. Since $|S'|=|S|+l-|S\cap E(C_{2l})|\leq |S|+l-1$, we have $$f(G,M\oplus E(C_{2l}))\leq |S'|\leq|S|+l-1=f(G,M)+l-1.$$ Hence $f(G,M\oplus E(C_{2l}))-f(G,M)\leq l-1$. When we convert $M\oplus E(C_{2l})$ to $M$, we obtain that $f(G,M)-f(G,M\oplus E(C_{2l}))\leq l-1$. Thus, we obtain the required result.
\end{proof}

Combining Theorem \ref{qpp}, Corollary \ref{qpp2} and Lemma \ref{mt}, we obtain the following result.
\begin{cor}\label{cor}Let $G$ be a bipartite graph of order $2n$ and with $F(G)= n-k$ where $1\leq k\leq n$. Then there are at most $k-1$ gaps between any two successive integers of the forcing spectrum of $G$.
If $G$ is $C_6$-, $C_8$-, \dots, and $C_{2(k+1)}$-free, then the forcing spectrum of $G$ is continuous.
\end{cor}
\begin{cor}If $G$ is a $C_6$-free bipartite graph of order $2n$ and with $F(G)\geq n-2$, then the forcing spectrum of $G$ is continuous.
\end{cor}

\section{\normalsize The minimum forcing number}
In this section we want to give a lower bound on the minimum forcing number of a bipartite graph when its maximum forcing number is large. Let $G$ be a bipartite graph of order $2n$. If   $F(G)\geq n-1$, then $f(G)=n-1$ by Corollary \ref{cor2.2}. Next we obtain the following result.

\begin{thm}\label{lower}
Let $G$ be a bipartite graph of order $2n$ and with $F(G)=n-2$. Then $$f(G)\geq \lceil\frac{n}{2}\rceil-1.$$
\end{thm}


To obtain such result, we need the following preliminary results.
\begin{lem}\cite{25} \label{fps}Let $G$ be a graph of order $2n$ and with a perfect matching. If the largest order of an induced subgraph of $G$ having a unique
perfect matching is $2k$, then  $f(G)\geq n-k$.
\end{lem}

\begin{cor}\label{cor3}Let $G$ be a bipartite graph of order $2n$ and with $F(G)= n-2$. Then $$f(G)\geq n-1-\delta(G).$$
\end{cor}
\begin{proof}
Let $M_s=\{u_1v_1,u_2v_2,\dots,u_nv_n\}$ be a perfect matching of $G$ with $f(G,M_s)=n-2$. Without loss of generality, we assume that $v_1$ is a vertex of the minimum degree. Since $G$ has a perfect matching, the degree of $v_1$ is at least 1. We may suppose that $\{u_1,u_2,\dots,u_{\delta(G)}\}$ are all neighbors of $v_1$. For $\delta(G)+1\leq i< j\leq n$, since $v_1$ is a vertex of degree 1 in the induced subgraph $G[\{u_1,v_1,u_i,v_i,u_j,v_j\}]$, by Lemma \ref{2.1}, $\{u_iv_j,v_iu_j\}$ is contained in $E(G)$. Let $G_2:=G[\{u_{\delta(G)+1},v_{\delta(G)+1},\dots,u_{n},v_{n}\}]$ and $G_1:=G-V(G_2)$. Then $G_2$ is a complete bipartite graph.

Let $G'$ be an induced subgraph of $G$ with a unique perfect matching $M'$. Then $G'$ is a bipartite graph. We will prove that the order of $G'$ is at most $2\delta(G)+2$.
If $G'$ contains no vertices of $u_{\delta(G)+1}$, $u_{\delta(G)+2}$, $\dots$, $u_{n}$, then $G'$ contains at most $2\delta(G)$ vertices. Otherwise, we suppose that the intersection of $\{u_{\delta(G)+1},u_{\delta(G)+2},\dots,u_{n}\}$ (resp. $\{v_{\delta(G)+1},v_{\delta(G)+1},\dots,v_{n}\}$) and the set of all endvertices of edges in $M'\cap E(V(G_1),V(G_2))$ is denoted by $\{y_1,y_2,\dots,y_b\}$ (resp. $\{x_1,x_2,\dots,x_w\}$). Thus there are $w+b$ edges in $M'\cap E(V(G_1),V(G_2))$ in total. Next we will find $b$ $M'$-alternating paths in $G$ separately starting with $y_1$, $y_2$, \dots, $y_b$ and $w$ $M'$-alternating paths separately starting with $x_1$, $x_2$, \dots, $x_w$ and prove that all such paths ending up with some $M'$-unsaturated vertices of $G_1$.

Since $y_l$ is $M'$-saturated, say $y_lv_i\in M'$, we have $v_i\in V(G_1)$. Along the edge $u_iv_i$, if $u_i$ is also $M'$-saturated and $u_iv_j\in M'$. Then $v_j\neq x_t$ for any $1 \leq t\leq w$. Otherwise, $v_j$ is some $x_t$. Then $y_lv_iu_ix_ty_l$ is an $M'$-alternating cycle, which contradicts that $M'$ is unique perfect matching of $G'$. Hence $v_j\in V(G_1)$. Continuing the process, we can find an $M'$-alternating path ending up with an $M'$-unsaturated vertex of $G_1$
 which is denoted by $\overline{y_l}$ (see an example in Fig. \ref{eps2}).
\begin{figure}[h]
\centering
\includegraphics[height=3.2cm,width=12cm]{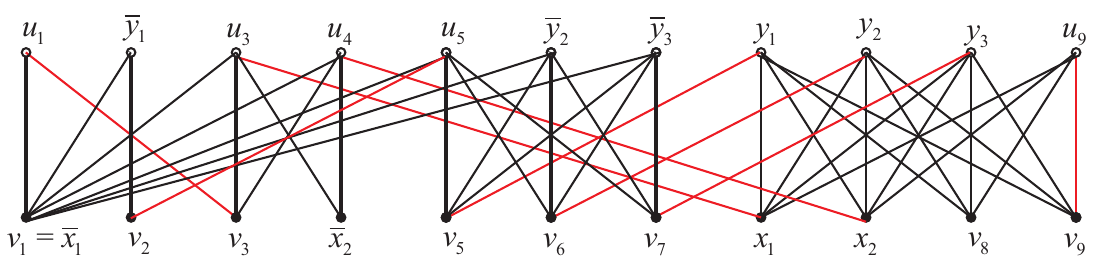}
\caption{\label{eps2}Illustration of $M'$-alternating paths in the proof of Corollary \ref{cor3}.}
\end{figure}

By a similar argument, we can find $w$ $M'$-alternating paths starting with $x_t$ and ending up with some $M'$-unsaturated vertex $\overline{x_t}$ of $G_1$ for $1\leq t\leq w$. Hence $G_1$ has at least $w+b$ vertices not in $G'$, and $G_1$ contains at most $2\delta(G)-w-b$ vertices of $G'$. Since $G'$ contains no $M'$-alternating cycles, $G_2$ contains at most one edge of $M'$ and thus $G_2$ contains at most $w+b+2$ vertices of $G'$. Thus, $G'$ has the order at most $2\delta(G)+2$.
By Lemma \ref{fps}, we have $f(G)\geq n-[\delta(G)+1]=n-1-\delta(G)$.
\end{proof}
\begin{lem}\cite{liu}\label{liu}
If $G$ is a bipartite graph with a perfect matching, then $f(G)\geq \delta(G)-1$.
\end{lem}

\noindent {\bf Proof of Theorem \ref{lower}} By Corollary \ref{cor3} and Lemma \ref{liu}, we have $$f(G)\geq \text{max}\{\delta(G)-1, n-1-\delta(G)\}.$$ If $\delta(G)\geq \frac{n}{2}$, then we have $\delta(G)-1\geq n-1-\delta(G)$ and thus $f(G)\geq \delta(G)-1\geq\frac{n}{2}-1$. Otherwise, we have $\delta(G)-1<n-1-\delta(G)$ and thus $f(G)\geq n-1-\delta(G)>\frac{n}{2}-1$. Thus, Theorem \ref{lower} holds. \hfill $\square$\\

In general we propose the following problem.
\begin{prob} \label{fps5}
Let $G$ be a bipartite graph of order $2n$ and with $F(G)\geq n-k$ for $1\leq k\leq n-1$. Does    $f(G)\geq \lceil\frac{n}{k}\rceil-1$ always hold?
\end{prob}

We also find this lower bound can be attained.

\begin{remark}For $n\equiv 1$ (mod $k$), let $H$ be a graph obtained from the disjoint union of $k-1$ copies of $K_{\lfloor\frac{n}{k}\rfloor,\lfloor\frac{n}{k}\rfloor}$ and one copy of $K_{\lceil\frac{n}{k}\rceil,\lceil\frac{n}{k}\rceil}$ and then adding a maximum matching between $A_i$ and $B_{i+1}$ for $1\leq i\leq k$, where $A_i$ and $B_i$ are the  partite sets of the $i$-th complete bipartite graph and $B_{k+1}=B_1$. Then $F(H)= n-k$ and $f(H)=\lceil\frac{n}{k}\rceil-1$.
\end{remark}
\begin{proof}
Let $M$ be a perfect matching of $H$ with $A_i$ matched to $B_i$ for $1\leq i\leq k$. Since any set of $k+1$ edges in $M$ contain two edges from the same complete bipartite graph and hence the subgraph induced by endvertices of such $k+1$ edges contains an $M$-alternating cycle. By Lemma \ref{2.1}, we have $F(H)\geq f(H,M)\geq n-k$. Next we prove that $f(H,F)\leq n-k$ for any perfect matching $F$ of $H$.  Let $T=\cup_{i=1}^kE(A_i,B_{i+1})$. If $F\cap T=\emptyset$, then let $S:=\cup_{i=1}^{k}\{e_i\}$, where $e_i\in F$ is one edge from $i$'th copy of $K_{\lfloor\frac{n}{k}\rfloor,\lfloor\frac{n}{k}\rfloor}$ for $1\leq i\leq k-1$, and $e_k\in F$ is the unique edge in $K_{\lceil\frac{n}{k}\rceil,\lceil\frac{n}{k}\rceil}$ such that it is not adjacent to any edges of $E(A_k,B_{1})$. Then $F\setminus S$ is a forcing set of $F$, and $f(H,F)\leq |F\setminus S|= n-k$. If $F\cap T\neq\emptyset$, then $F\cap E(A_i,B_{i+1})\neq \emptyset$ for $1\leq i\leq k$ since each complete bipartite graph is of even order. Let $e'_i$ be an edge of $F\cap E(A_i,B_{i+1})$ for $1\leq i\leq k-1$ and $e'_k\in F$ is an edge of $K_{\lceil\frac{n}{k}\rceil,\lceil\frac{n}{k}\rceil}$. Let $S':=\cup_{i=1}^{k}\{e'_i\}$. Then $F\setminus S'$ is a forcing set of $F$, and $f(H,F)\leq |F\setminus S'|= n-k$. By the arbitrariness of $F$, we have $F(H)\leq n-k$.

The union of these maximum matchings saturates all vertices but two of $K_{\lceil\frac{n}{k}\rceil,\lceil\frac{n}{k}\rceil}$. Since the two unsaturated vertices are adjacent,  these maximum matchings can be extended to a perfect matching $M'$ of $H$. Since fixing the edges between $A_1$ and $B_2$ fixes the perfect matching $M'$, we have $f(H)\leq f(H,M')\leq \lceil\frac{n}{k}\rceil-1$.
Note that $\delta(H)=\lceil\frac{n}{k}\rceil$. By Lemma \ref{liu}, $f(H)\geq \delta(H)-1=\lceil\frac{n}{k}\rceil-1$. Therefore, $f(H)=\lceil\frac{n}{k}\rceil-1$.
\end{proof}

\begin{remark}
    For  non-bipartite graphs the lower bound in Problem \ref{fps5} is too larger. Take graphs $H_l$ of order $2n$ and with $F(H_l)=n-1$ in \cite{47} for example where $1\leq l\leq \frac{n-1}{2}$. By Lemma 5.6 in \cite{47}, we obtain that $f(H_l)=n-l-1\leq n-2< n-1$.
\end{remark}
Finally we  discuss the distribution of minimum forcing number for such classes of bipartite graphs. For integers $0\leq l\leq \lfloor\frac{n}{2}\rfloor-1$ and $n\geq 2$, let $G_l$ be a bipartite graph with  a perfect matching $M_s=\{u_iv_i|i=1,2,\dots,n\}$  such that $G^1=G_l[\{u_i,v_i|i=1,2,\dots,\lceil\frac{n}{2}\rceil+l\}]$ and $G^2=G_l-V(G^1)$ are two complete bipartite graphs, and the edges between  $G^1$ and $G^2$ are as follows ($x=2l+3$ if $n$ is odd and $x=2l+2$ otherwise) (see Fig. \ref{eps41}): $$\{u_xv_{\lceil\frac{n}{2}\rceil+l+1},v_xu_{\lceil\frac{n}{2}\rceil+l+1},u_{x+1}v_{\lceil\frac{n}{2}\rceil+l+2},v_{x+1}u_{\lceil\frac{n}{2}\rceil+l+2},\dots,u_{\lceil\frac{n}{2}\rceil+l}v_{n-1},
v_{\lceil\frac{n}{2}\rceil+l}u_{n-1}\}.$$ Then $G_l$ is a bipartite graph with $F(G)=n-2$.
\begin{figure}[h]
\centering
\includegraphics[height=4cm,width=16cm]{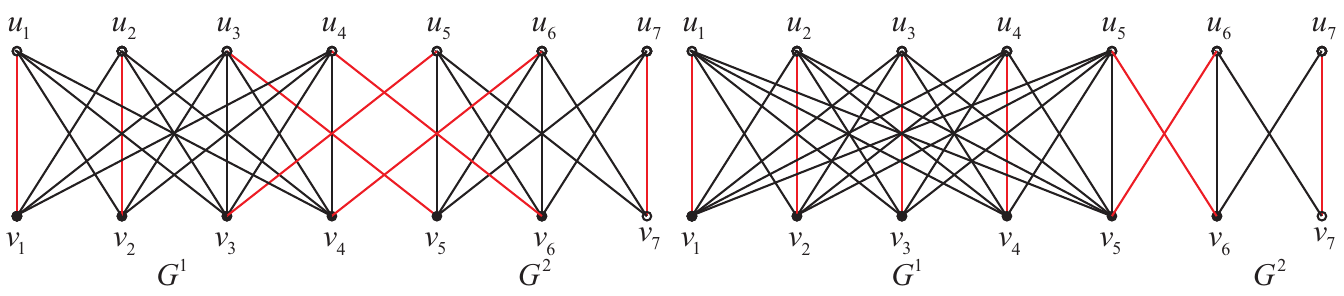}
\caption{\label{eps41}Examples of $G_l$ where $n=7$ and $l=0$ (left) and 1 (right).}
\end{figure}

\begin{lem}\label{lem3}For $0\leq l\leq \lfloor\frac{n}{2}\rfloor-1$, we have $f(G_l)=\lceil\frac{n}{2}\rceil+l-1$.
\end{lem}
\begin{proof}Note that the degree of $u_n$ is $n-(\lceil\frac{n}{2}\rceil+l)=\lfloor\frac{n}{2}\rfloor-l$. Since $u_n$ is a vertex of $G_l$ with the minimum degree, $\delta(G_l)=\lfloor\frac{n}{2}\rfloor-l$. By Corollary \ref{cor3}, we have $$f(G_l)\geq n-1-\delta(G_l)=n-1-\lfloor\frac{n}{2}\rfloor+l=\lceil\frac{n}{2}\rceil+l-1.$$

Let $M=E(V(G^1),V(G^2))\cup\{u_1v_1,\dots,u_{x-1}v_{x-1},u_nv_n\}$ (shown in Fig. \ref{eps41} where red lines form $M$). Then $M$ is a perfect matching of $G_l$.
Since $$G_l[\{u_1,v_1,u_x,v_{\lceil\frac{n}{2}\rceil+l+1},u_{x+1},v_{\lceil\frac{n}{2}\rceil+l+2},\dots,u_{\lceil\frac{n}{2}\rceil+l},v_{n-1},u_n,v_n\}]$$ has a unique perfect matching, $\{u_2v_2,\dots,u_{x-1}v_{x-1},v_xu_{\lceil\frac{n}{2}\rceil+l+1},v_{x+1}u_{\lceil\frac{n}{2}\rceil+l+2},\dots,v_{\lceil\frac{n}{2}\rceil+l}u_{n-1}\}$ is a forcing set of $M$. So we have $f(G_l)\leq f(G_l,M)\leq \lceil\frac{n}{2}\rceil+l-1$.
\end{proof}

By Lemma \ref{lem3}, we can obtain the following result immediately.
\begin{cor}All minimum forcing numbers of the bipartite graphs $G$ of order $2n$ and with $F(G)=n-2$ form an integer interval $[\lceil\frac{n}{2}\rceil-1,n-2]$.
\end{cor}



\smallskip
\smallskip

\noindent{\normalsize \textbf{Acknowledgments}}
\smallskip
\smallskip

This work is supported by National Natural Science Foundation of China (Grant No. 12271229), Natural Science Foundation of Inner Mongolia Autonomous Region of China (Grant No. 2025QN01008), and first-class discipline research special project (Grant No. YLXKZX-NGD-055).



\end{document}